\documentclass{svproc}
\usepackage{latexsym, amsmath, amstext, amssymb, amsfonts, amscd, bm, array, multirow, amsbsy, mathrsfs}
\usepackage{graphicx,scalerel}
\newcommand\wh[1]{\hstretch{2}{\hat{\hstretch{.5}{#1}}}}
\usepackage{url}
\newcommand{\Conf}{\mathrm{Conf}}

\def\dd{\mathrm{d}}
\def\H{\mathrm{H}}
\def\h{\mathfrak{h}}

\def\cN{\mathcal{N}}
\def\cL{\mathcal{L}}
\def\ric{\mathrm{Ric}}
\def\tm{\wh{M}}
\def\tg{\wh{g}}

\usepackage{graphicx}

\begin{document}
\mainmatter              
\title{$\varepsilon\,$-contact structures and six-dimensional supergravity }
\titlerunning{$\varepsilon\,$-contact structures and six-dimensional supergravity}  
%
\author{\'Angel Jes\'us Murcia Gil\inst{1,2}}
\authorrunning{\'Angel Murcia} 
%
\tocauthor{\'Angel Murcia}
\institute{Instituto de F\'isica Te\'orica UAM/CSIC, C/ Nicol\'as Cabrera, 13-15, C.U. Cantoblanco, E-28049 Madrid, Spain,\\
\email{angel.murcia@csic.es},\\ 
\and
Department of Mathematics, University of Hamburg. Bundesstra$\beta$e 55, D-20146 Hamburg, Germany
}

\maketitle              

\begin{abstract}
We introduce the concept of $\varepsilon\,$-contact metric structures on oriented (pseudo-)Riemannian three-manifolds, which encompasses the usual Riemannian contact metric, Lorentzian contact metric and para-contact metric structures, but which also allows the possibility for the Reeb vector field to be null. We investigate in more detail this latter case, which we call null contact structure. We observe that it is possible to extend in a natural and meaningful way both the Sasaki and K-contact conditions for null-contact structures, but we find that they are not equivalent conditions, in contradistinction to the situation for non-lightlike Reeb vector fields. Finally, we define the notion of $\varepsilon\eta\,$-Einstein structures and we discover that appropriate direct products of these structures produce solutions of six-dimensional minimal supergravity coupled to a tensor multiplet with constant dilaton. 
\end{abstract}
\section{Introduction}

Some of the most challenging problems which are being currently studied in Differential Geometry were originally encountered  in the framework of String Theory or supergravity \cite{Cecotti,FreedmanProeyen,Ortin}. Just to mention a few, we may think of mirror symmetry \cite{Lerche,Dixon,MS}, the Strominger system \cite{Strominger,LiYau,Mario} or the the classification of all simply-connected manifolds which admit bosonic supergravity solutions in a certain dimension \cite{Granya,Fei,Mario2}.  

This latter problem is the one we will focus our efforts on. More concretely, we will intend to help clear the panorama for the particular case of  six-dimensional minimal supergravity coupled to a tensor multiplet with constant dilaton \cite{Nishino:1984gk,Nishino:1986dc}. If $\tm$ is a six-dimensional manifold, the bosonic configuration space $\Conf(\tm)$ of this theory is formed by all pairs $(\wh{g}, \mathrm{H})$ of Lorentzian metrics $\wh{g}$ on $\tm$ and three-forms $\H \in \Omega^3(\tm)$. An element $(\wh{g},\H) \in \Conf(\tm)$ is a solution of such theory if the following equations hold:
\begin{equation}
\label{eq:sugraeqintro}
\mathrm{Ric}^{\wh{g}}-\frac{1}{4} \H \circ \H=0\, , \quad \dd \H=0 \, , \quad \dd \star_{\wh{g}} \H=0 \, , \quad \vert \H \vert_{\wh{g}}^2=0\,,
\end{equation}
where $\mathrm{Ric}^{\wh{g}}$ is the Ricci curvature tensor of $\wh{g}$, $\star_{\wh{g}}: \Omega^3(\tm) \rightarrow \Omega^3(\tm)$ is the Hodge star map and $\H \circ \H$ is a symmetric $(0,2)$-tensor defined by $(\H \circ \H)(X,Y)=\wh{g}^{-1}(\iota_X H,\iota_Y \H)$ for every $X, Y \in \mathfrak{X}(\tm)$. With these provisos in mind, we shall look for solutions $(\wh{g}, \H) \in \Conf(\tm)$ to equation (\ref{eq:sugraeqintro}) assuming the following direct-product ansatz:
\begin{equation}
(\tm,\wh{g})=(N \times X, \chi \oplus h)\, , 
\end{equation}
where $(N,\chi)$ is 3-dimensional oriented Lorentzian manifold and $(X,h)$ a 3-dimensional oriented Riemannian manifold. After imposing an ansatz for $\H$ consistent with this splitting, the initially six-dimensional problem is in turn divided into two three-dimensional problems. 

It is precisely at this point where $\varepsilon\,$-contact metric structures arise. They are defined over a three-dimensional oriented manifold and encapsulate the usual notions of Riemannian contact metric structures, Lorentzian contact metric structures and para-contact metric structures \cite{Blair,Calvaruso,CalvarusoII} when the aforementioned three-manifold is equipped with a metric of Riemannian or Lorentzian signature, correspondingly. However, these $\varepsilon\,$-contact metric structures include a fourth type of structures, which we call null contact metric structures. They are defined over Lorentzian manifolds and are characterized by having a null Reeb vector field. These null contact metric structures do not seem, to the best of our knowledge, to have been previously explored in the literature and we have found them to enjoy fairly intriguing properties, such as the corresponding Sasaki condition \cite{Sasaki} for null contact structures not being equivalent to that of K-contactness \cite{Blair}.  

In the context of $\varepsilon\,$-contact metric structures, we introduce the notion of $\varepsilon\eta\,$-Einstein structures, characterized by possessing a Ricci curvature tensor with a prescribed given structure. For non-null Reeb vector fields, they are particular cases of the standard $\eta\,$-Einstein structures \cite{Okumura,Boyer} and their definition is justified because, as it is proven in Theorem \ref{th:elteorema}, it is possible to intertwine $\varepsilon\eta\,$- structures on $(N,\chi)$ and $(X,h)$ respectively to yield a direct-product six-dimensional manifold $(\tm,\wh{g})$ admitting solutions of six-dimensional supergravity.


The outline of the document is as follows. First, we introduce the concept of $\varepsilon\,$-contact metric structure and study some of its most important properties. Secondly, we focus on null contact structures, specifying their more characteristic features and showing that they admit meaningful notions for the Sasaki and K-contact conditions. Finally, we define $\varepsilon\eta\,$-Einstein structures and show that they can be used for the construction of solutions of six-dimensional minimal supergravity coupled to a tensor multiplet with constant dilaton.

The contents of this contribution are based on a previous work \cite{Angel} of the author together with C.S. Shahbazi. 
\subsubsection*{Acknowledgements.}
I acknowledge support from the Deutscher Akademischer Austauschdienst (DAAD) through the Short-Term Research Grant No. 91791300  and from the Spanish FPU Grant, No.  FPU17/04964. I was also indirectly supported by the MCIU/AEI/FEDER UE grant PGC2018-095205-B-I00 and by the ``Centro de Excelencia Severo Ochoa'' Program grant SEV-2016-0597. Finally, I wish to thank C.S. Shahbazi for useful comments and my parents for their permanent support.

On the other hand, this is a preprint of the following chapter: \'Angel Murcia, $\varepsilon\,$-contact structures and six-dimensional supergravity, published in ``Developments in Lorentzian Geometry", edited by Alma L. Albujer, Magdalena Caballero, Alfonso García-Parrado, Jónatan Herrera and Rafael Rubio, 2022, Springer reproduced with permission of Springer Nature Switzerland. The final authenticated version is available online at: https://doi.org/10.1007/978-3-031-05379-5. Please access this link:\\ \url{https://link.springer.com/book/10.1007/978-3-031-05379-5}

\section{$\varepsilon\,$-contact metric structures}

In this section we will present the definition of $\varepsilon\,$-contact metric structures and show that it encompasses the usual notions of Riemannian contact metric structures, Lorentzian contact metric structures and para-contact metric structures. We will study also some of the most relevant features of $\varepsilon\,$-contact metric structures and define the Sasaki and K-contact conditions.
\begin{definition}
Let $(M,g)$ be an oriented Riemannian or pseudo-Riemannian three-manifold. An $\varepsilon\,$-contact metric structure (or just $\varepsilon\,$-contact structure) on $M$ is a triple $(g,\alpha, \varepsilon)$, with $\varepsilon \in \{-1,0,1\}$ and $\alpha$ a one-form $\alpha \in \Omega^1(M)$ such that:
\begin{equation}
\alpha= \star_g \dd \alpha\, , \quad \vert \alpha \vert_g^2=\varepsilon\,, \label{eq:defcontact}
\end{equation}
where $\star_g: \Omega^{r}(M) \rightarrow \Omega^{3-r}(M)$ $(r=0,1,2,3)$ denotes the Hodge dual with respect to $g$ and the orientation fixed on $M$, which in turn is said to be an $\varepsilon\,$-contact (metric) three-manifold. If $g$ is Lorentzian, we will assume it to be oriented and time-oriented. 
\end{definition}
\begin{remark}
When denoting $\varepsilon\,$-contact structures $(g,\alpha,\varepsilon)$, we might drop sometimes the number $\varepsilon$ whenever it is clear from the context its precise value and no confusion may arise. Also, owing the fact that we will always assume the presence of a (pseudo-)Riemannian metric, we will use indistinctly the nomenclature $\varepsilon\,$-contact metric structure and $\varepsilon\,$-contact  structure. Finally, note that the definition given here of $\varepsilon\,$-contact metric structures refers only to three dimensions, so unless otherwise stated, we will always suppose that we work with three-dimensional manifolds. 
\end{remark}

\begin{remark}
The equation $\alpha= \star_g \dd \alpha$ can be equivalently expressed as:
\begin{equation}
\star_g \alpha= s_g  \dd \alpha\,,
\end{equation}
where $s_g=1$ if $g$ is Riemannian and $s_g=-1$ if $g$ is Lorentzian. 
\end{remark}
Given any $\varepsilon\,$-contact structure $(g,\alpha,\varepsilon)$ on $M$, we define the Reeb vector field $\xi \in \mathfrak{X}(M)$ and the endomorphisms $\phi: TM \rightarrow TM$ and $\mathfrak{h}: TM \rightarrow TM$ as follows:
\begin{equation}
\xi=\alpha^\sharp\, , \quad \phi(v)=-s_g (\iota_v \star_g \alpha)^\sharp\, , \quad \mathfrak{h}(v)=(\mathcal{L}_\xi \phi )(v)  \quad \forall v \in TM \,, \label{eq:defphi}
\end{equation}
where $\sharp: T^*M \rightarrow TM$ denotes the musical isomorphism defined by $g$, $\iota_v$ the interior product with $v$ and $\mathcal{L}$ the Lie derivative.
\begin{lemma}
\label{lemma:1}
Let $(g,\alpha,\varepsilon)$ be an $\varepsilon\,$-contact structure on $M$. Then the following identities hold:
\begin{eqnarray}
\label{eq:lemma11}
g(\mathrm{Id}\otimes \phi)=-g(\phi \otimes \mathrm{Id})=\dd \alpha \, ,& \quad &\phi(\xi)=0\, , \quad \alpha \circ \phi=0\, , \\ \phi^2=s_g (-\varepsilon\, \mathrm{Id}+ \xi \otimes \alpha) \, , & &g (\phi \otimes \phi)=s_g (\varepsilon g -\alpha \otimes \alpha)\,,\label{eq:lemma12}
\end{eqnarray}
where $\mathrm{Id}:TM\rightarrow TM$ denotes the identity map. 
\end{lemma}
\begin{proof}
By equation (\ref{eq:defphi}) it is clear that:
\begin{equation}
\begin{split}
g(v_1, \phi(v_2))&=-s_g g(v_1,(\iota_{v_2} \star_g \alpha)^\sharp )=-s_g \star_g \alpha(v_2,v_1)\\&=\dd \alpha(v_1,v_2)=-\dd \alpha(v_2,v_1)=-g(\phi(v_1),v_2)\, , \quad v_1,v_2 \in \mathfrak{X}(M)\,.
\end{split}
\end{equation} 
This proves the first equation in (\ref{eq:lemma11}). Similarly, since for any 1-form $\eta \in \Omega^1(M)$ we have that $\iota_{\eta^\sharp} \star_g \eta=0$, the second and third equations in (\ref{eq:lemma11}) follow. Regarding the first equation in (\ref{eq:lemma12}):
\begin{equation}
\begin{split}
\phi (\phi (v))=[\iota_{(\iota_{v} \star_g \alpha)^\sharp} (\star_g \alpha)]^\sharp&=[(\star_g \alpha)((\iota_{v} \star_g \alpha)^\sharp)]^\sharp=[\star_g(\alpha \wedge (\iota_{v} \star_g \alpha) )]^\sharp\\&=[-\star_g (\iota_{v} (\alpha \wedge \star_g \alpha))+\alpha(v) (\star_g \star_g \alpha)]^\sharp\\&=-s_g \varepsilon \, v+s_g \alpha (v) \xi\, , \qquad v\in \mathfrak{X}(M)\, 
\end{split}
\end{equation}
Finally, the second equation in (\ref{eq:lemma12}) is proven upon the use of the previous expressions:
\begin{equation}
g(\phi(v_1),\phi(v_2))=-g(v_1,\phi^2(v_2))=s_g \varepsilon g(v_1,v_2)-s_g \alpha(v_1) \alpha(v_2)\, , \quad v_1, v_2 \in \mathfrak{X}(M)\,.
\end{equation}
\end{proof}
Using Lemma \ref{lemma:1}, we are able to prove:
\begin{proposition}
Let $M$ be an oriented three-manifold. An $\varepsilon\,$-contact metric structure $(g,\alpha,\varepsilon)$ on $M$ defines a Riemannian contact metric structure if $g$ is Riemannian, a Lorentzian contact metric structure if $g$ is Lorentzian and $\varepsilon=-1$ or a para-contact metric structure if $g$ is Lorentzian and $\varepsilon=1$. The converse is also true for the three previous cases. 
\end{proposition}
\begin{proof}
By Lemma \ref{lemma:1}, we observe that the defining conditions of a Riemannian contact metric structure \cite{Blair} are satisfied if $g$ is Riemannian. Similarly, if $g$ is Lorentzian and $\varepsilon=-1$, we learn that we have a Lorentzian contact structure by direct comparison to the usual definition \cite{Calvaruso}. Finally, if $g$ is Lorentzian and $\varepsilon=1$ we conclude that $(g,\alpha)$ defines a para-contact structure on $M$ after a careful look of the standard definition \cite{CalvarusoII}. The converse for these three cases can be seen to be true by reconstructing, through the appropriate use of (\ref{eq:lemma11}) and (\ref{eq:lemma12}), the equations in (\ref{eq:defcontact}). 
\end{proof}
On the other hand, it is always possible to define a special (local) frame which is highly convenient for computations. We call it $\varepsilon\,$-contact frame.
\begin{definition}
Let $(g,\alpha,\varepsilon)$ be an $\varepsilon\,$-contact structure. An $\varepsilon\,$-contact frame is a local frame $\{\xi, u, \phi(u)\}$, with $u$ a locally-defined vector field which does not vanish in its domain of definition, such that:
\begin{equation}
g(u,\xi)=1-\varepsilon^2\, , \quad g(u,u)=s_g \varepsilon\,.
\end{equation}
In the case $\varepsilon=0$, we will refer to an $\varepsilon\,$-contact frame as a light-cone frame. 
\end{definition}
Note that $\varepsilon\,$-contact frames always exist, at least locally. An $\varepsilon\,$-contact frame satisfies:
\begin{equation}
g(\xi,\xi)=\varepsilon\, , \quad g(\xi, \phi(u))=0\, , \quad g(u,\phi(u))=0\,, \quad g (\phi (u), \phi(u))=1\,.
\end{equation}
Some additional properties of $\varepsilon\,$-contact structures are given by the following Proposition: 
\begin{proposition}
Let $(g,\alpha,\varepsilon)$ be an $\varepsilon\,$-contact structure on $M$. Then:
\begin{equation}
\nabla_\xi \xi=0\, , \quad \nabla_\xi \phi=0\, , \quad \mathfrak{h}(\xi)=0\, , \quad \mathrm{Tr}(\mathfrak{h})=0\, , \quad \mathcal{L}_\xi \alpha=0\, , \quad \mathfrak{h} \circ \phi=-\phi \circ \mathfrak{h}\,,
\end{equation}
where $\nabla$ is the Levi-Civita connection of $g$. Furthermore, $\mathfrak{h}$ is symmetric with respect to $g$. 
\end{proposition} 
\begin{proof}
Since $\vert \alpha \vert_g^2=\varepsilon$ and $\alpha= \star_g \dd \alpha$, then we directly have that $\mathcal{L}_\xi \alpha=0$. Then:
\begin{equation}
0=\mathcal{L}_\xi \alpha (v)=\xi (g(v,\xi))-g(\nabla_\xi v-\nabla_v \xi,\xi)=g(v,\nabla_\xi \xi) \, , \quad v \in \mathfrak{X}(M)\,.
\end{equation}
Consequently $\nabla_\xi \xi=0$. Applying now $\nabla_\xi$ on both sides of $g(\mathrm{Id} \otimes \phi)=s_g \star_g \alpha$ we see that $\nabla_\xi \phi=0$. On the other hand, $\mathfrak{h}(\xi)=0$ follows trivially from the definition and applying $\mathcal{L}_\xi$ on the expression for $\phi^2$ in (\ref{eq:lemma12}) we get straightforwardly $\mathfrak{h} \circ \phi =-\phi \circ \mathfrak{h}$. Choosing an $\varepsilon\,$-contact frame $\{\xi,u,\phi(u)\}$, for $\varepsilon \neq 0$ we have:
\begin{equation}
\mathrm{Tr}(\mathfrak{h})=s_g \varepsilon g(u,\mathfrak{h} (u))+g(\phi(u),\mathfrak{h}\circ  \phi(u))=s_g \varepsilon g(u,\mathfrak{h} (u))+g(u,\mathfrak{h}\circ  \phi^2 (u))=0\,.
\end{equation}
For $\varepsilon=0$:
\begin{equation}
\mathrm{Tr}(\mathfrak{h})=g(\xi,\mathfrak{h} (u))+g(\phi(u),\mathfrak{h}\circ  \phi(u))=g(\xi,\mathfrak{h} (u))=g(\xi, -\nabla_{\phi(u)} \xi+\phi (\nabla_u \xi) )=0\,,
\end{equation}
where we have used that $g(\phi(u),\mathfrak{h}\circ  \phi(u))=g(u,\mathfrak{h}\circ  \phi^2(u))=0$ since $\phi^2(u)=-\xi$ when $\varepsilon=0$. Finally, the symmetry of $\mathfrak{h}$ for all $\varepsilon$ follows from observing that $g(\xi,\mathfrak{h} (u))=g(\xi, -\nabla_{\phi(u)} \xi+\phi (\nabla_u \xi) )=0$, $g(\xi,\mathfrak{h}\circ \phi(u))=0$ and $g(u,\mathfrak{h} \circ \phi(u))=g(\mathfrak{h}(u),\phi(u))$.
\end{proof}
\begin{remark}
In this contribution we follow the conventions in which the exterior derivative $\dd \omega$ of any $p$-form $\omega $ takes the form:
\begin{eqnarray*}
\dd \omega(X_0,\dots, X_p)=\sum_{i} (-1)^i X_i(\omega(X_0,\dots, \hat{X}_i, \dots, X_i))\\
+ \sum_{i<j} (-1)^{i+j} \omega([X_i,X_j],X_0, \dots, \hat{X}_i, \dots, \hat{X}_j, \dots, X_p)\,.
\end{eqnarray*}
Much of the literature on contact geometry, see for example \cite{Blair}, uses the conventions of Kobayashi and Nomizu \cite{KN}, in which the formula of the exterior derivative differs by a factor of $\frac{1}{p+1}$ from the one stated above. 
\end{remark}
We finish this section by presenting the notions of Sasaki and K-contact $\varepsilon\,$-contact metric structures.
\begin{definition}
\label{def:sask}
Let $(g,\alpha,\varepsilon)$ be an $\varepsilon\,$-contact metric structure on $M$. It is said to be Sasakian (or simply Sasaki) if $\mathfrak{h}=0$. Similarly, it is said to be K-contact if the Reeb vector field is Killing, $\mathcal{L}_\xi g=0$.
\end{definition}
In 3-dimensions, it is known that the Sasaki and K-contact conditions are equivalent for $\varepsilon\,$-contact structures with $\varepsilon \neq 0$. However, we will see in the next section that this is no longer true when $\varepsilon=0$.

\section{Null contact metric structures}

We devote this section to the study of $\varepsilon\,$-contact structures with $\varepsilon=0$, which we simply call \emph{null contact structures}. To the best of our knowledge, these structures do not seem to have been previously explored in the literature. They are qualitatively different from the other $\varepsilon\,$-contact structures, since if $(g,\alpha)$ is a null contact structure:
\begin{equation}
\alpha \wedge \dd \alpha=-\alpha \wedge \star_g \alpha=0\,,
\end{equation}
where we have used that $\vert \alpha \vert_g^2=0$. This implies that the 1-form $\alpha$ defining the null contact structure is \emph{not} a contact form. However, since null contact structures satisfy that $\iota_\xi \dd \alpha=0$ and  the concept of $\varepsilon\,$-contact structures encompass Riemannian contact, Lorentzian contact and para-contact metric structures, it is natural to think of null contact structures as a generalization in which the Reeb vector field is null. Furthermore, we will see later that it is possible to introduce reasonable notions of Sasakianity and K-contactness, in analogy to the $\varepsilon \neq 0$ cases.

We start by pinpointing some characteristic properties of null contact structures.
\begin{proposition}
\label{prop:propnc}
Let $(g,\alpha)$ be a null contact structure. Then:
\begin{equation}
\phi^2=-\xi \otimes \alpha\, , \quad \phi^3=0\,, \quad \mathfrak{h} \circ \phi=\phi \circ \mathfrak{h}=0\,, \quad \mathfrak{h}=\mu\,  \xi \otimes \alpha\,,
\end{equation}
for a function $\mu \in C^\infty(M)$.
\end{proposition}
\begin{proof}
The first two equations follow directly from Lemma \ref{lemma:1} after setting $\varepsilon=0$. Regarding the third equation, let us note the identity $\dd \alpha(v_1,\phi(v_2))=-\alpha(v_1)\alpha(v_2)$ for every $v_1,v_2 \in \mathfrak{X}(M)$ and apply $\mathcal{L}_\xi$ at both sides:
\begin{equation}
\begin{split}
\dd \alpha( \cL_{\xi} v_1, \phi(v_2))+\dd \alpha(v_1,\mathfrak{h}(v_2))&+\dd \alpha(v_1, \phi(\cL_{\xi}v_2))=\\&-\alpha(\cL_{\xi}v_1)\alpha(v_2)-\alpha(v_1)\alpha(\cL_{\xi}v_2)\, .
\end{split}
\end{equation}
Therefore $\dd\alpha(v_1,\mathfrak{h}(v_2))=g(v_1,\phi \circ \h(v_2))=0$, so $\h \circ \phi=\phi \circ \h=0$. From here, since $\ker(\phi)=0$, $\h=\mu \xi \otimes \alpha$ for a certain function $\mu \in C^\infty(M)$ and we conclude.
\end{proof}
\begin{example}
Take $(M,g)=(\mathbb{R}^3,\delta)$ where $\delta=\mathrm{diag}(-1,1,1)$ is the Minkowski metric. Consider the 1-form $\alpha=e^y q(x-t) (\dd t-\dd x)$, which is globally-defined if we assume that the function $q \in C^\infty(\mathbb{R})$ is smooth everywhere and has no zeros. Then $(\delta,\alpha)$ defines a null contact structure. The Reeb vector field is $\xi=- e^y q(t-x)(\partial_t+\partial_x)$ and we have that:
\begin{equation}
\begin{split}
\mathcal{L}_\xi \delta&=2e^y q'(t-x) \left ( \dd t \otimes \dd t-  \dd t \odot \dd x+   \dd x \otimes \dd x\right)\\&+e^y q(t-x) \left (\dd t \odot \dd y -\dd x \odot \dd y  \right ) \,,
\end{split}
\end{equation}
where $\odot$ denotes the symmetric tensor product. Hence it is clearly not K-contact. On the other hand, the endomorphism $\phi$ reads:
\begin{equation}
\phi=e^y q(t-x)\left ( \partial_t \otimes \dd y  +\partial_y \otimes \dd t+ \partial_x \otimes\dd y -\partial_y \otimes \dd x \right)\,.
\end{equation}
Consequently, by direct computation:
\begin{equation}
\mathfrak{h}=\mathcal{L}_\xi \phi=e^{2y} (q(t-x))^2 (\partial_t \otimes \dd t+\partial_x \otimes \dd t-\partial_t  \otimes \dd x-\partial_x  \otimes \dd x)=-\xi \otimes \alpha\,.
\end{equation}
Therefore $(\delta,\alpha)$ is not Sasakian either.
\end{example}
\subsection{Sasakian and K-contact null contact structures}

The Sasaki and K-contact conditions were already defined for $\varepsilon\,$-contact structures back at Definition \ref{def:sask}. Regarding the Sasakian condition, for $\varepsilon\,$-contact structures with $\varepsilon\neq 0$ it is known \cite{Blair,Calvaruso} that it is equivalent to the existence of a certain integrable endomorphism in the tangent bundle $T(M\times \mathbb{R})$  whose square equals $-\varepsilon s_g \mathrm{Id}$, where $\mathrm{Id}$ is the identity operator. Interestingly enough, we are going to see next that this result extends naturally for null contact structures.



For that, let us define:
\begin{equation}
J: T (M\times \mathbb{R}) \rightarrow  T (M\times \mathbb{R})\, , \quad (v, c\, \partial_t) \mapsto (\phi(v)+ c\xi, \alpha(v) \partial_t)\, ,
\end{equation}
where $t$ is the canonical coordinate on $\mathbb{R}$ and $c \in \mathbb{R}$. By direct computation we check that $J^2=0$. 
\label{def:integra}
\begin{definition}
Let $E \in \Gamma(\mathrm{End}(TN))$ be a field of endomorphisms on a manifold $N$. It is said to be integrable if around each point $n \in N$ there exists a coordinate system on which the matrix representation of $E$ has constant coefficients.
\end{definition}
Note that, for almost complex structures, the usual notion of integrability is equivalent to the one given at Definition \ref{def:integra}. By a result of G. Thompson \cite{Thompson}, a given field of endomorphisms $J \in  \Gamma(\mathrm{End}(T(M \times \mathbb{R})))$ is integrable if and only if the following three conditions hold simultaneously:
\begin{itemize}
\item The Nijenjuis tensor $\mathcal{N}_J$ of $J$, defined as:
\begin{equation}
\begin{split}
\mathcal{N}_J(v_1,v_2)&=[J(v_1),J(v_2)]-J[v_1,J(v_2)]-J[J(v_1),v_2]\\&+J^2[v_1,v_2]\,, \quad v_1, v_2 \in T(M\times \mathbb{R})
\end{split}
\end{equation}
vanishes.
\item $J$ is a zero-deformable field of endomorphisms, i.e. around every point there exists a frame relative to which the Jordan form of this endomorphism is constant.
\item The distribution $\ker(J) \subset T(M \times \mathbb{R})$ is involutive. 
\end{itemize} 
\begin{proposition}
A null contact structure $(g,\alpha)$ is Sasakian if and only if the associated endomorphism $J: T (M\times \mathbb{R}) \rightarrow  T (M\times \mathbb{R})$ is integrable. 
\end{proposition}
\begin{proof}
Assume first that $(g,\alpha)$ is Sasakian. If $\{\xi, u, \phi(u)\}$ denotes a local light-cone basis, let $\{\xi, u, \phi(u), \partial_t\}$ be a local frame on $T(M\times \mathbb{R})$. In this basis $J$ has the matrix representation:
\begin{equation}
J =
\left( {\begin{array}{cccc}
	0 & 0 & -1 & 1 \\
	0 & 0 & 0 & 0 \\
	0 & 1 & 0 & 0 \\
	0 & 1 & 0 & 0 \\
	\end{array} } \right)\,.
\end{equation}
Since at every point there exist local frames $\{\xi, u, \phi(u), \partial_t\}$, this  proves $J$ is zero deformable. On the other hand, it is clear that:
\begin{equation}
\ker (J)= \mathrm{Span}_{C^\infty} (\xi, \phi(u)+\partial_t )\,.
\end{equation}
However, we observe that $[\xi, \phi(u)+\partial_t]=[\xi, \phi(u)]=\phi([\xi,u])+\mathfrak{h}(u)=\kappa \xi$ for some $\kappa \in C^\infty(M)$, since by $\mathcal{L}_\xi \alpha=0$ we have that $\alpha([\xi,u])=0$. This implies $\ker (J)$ is involutive. Similarly, after some computations:
\begin{eqnarray*}
&\cN_J(\xi,u) = - J[\xi,\phi(u)] = 0\, , \quad \cN_J(\xi,\phi(u)) = - J[\xi,J(\phi(u))] = 0\, , \\ 
&\cN_J(u,\phi(u)) = - J[J(\phi(u)),J(\phi(u))] = 0\, , \quad  \cN_J(\xi,\partial_t) = -J[\xi,J(\partial_t)] = 0\, ,\\
&\cN_J(u,\partial_t) =  [\phi(u),\xi] - J[u,\xi] = - \cL_{\xi}(\phi(u)) + J(\cL_{\xi}u) \\&= - \cL_{\xi}(\phi(u)) + \phi(\cL_{\xi}u) = -\mathfrak{h}(u) = 0\, , \\
&\cN_J(\phi(u),\partial_t) = [J(\phi(u)), J(\partial_t)] - J[\phi(u), J(\partial_t)]\\& = [\phi^2(u), \xi] - J[\phi(u), \xi] = \phi^2(\cL_{\xi} u) = 0\, .
\end{eqnarray*}
Consequently $\cN_J$ is identically zero and therefore $J$ is integrable. To prove the converse, let us assume $J$ is integrable. Then $\cN_J=0$ and similarly as above, we compute:
\begin{equation}
\begin{split}
\cN_J(u,\partial_t) &=  [\phi(u),\xi] - J[u,\xi] = - \cL_{\xi}(\phi(u)) + J(\cL_{\xi}u) \\&= - \cL_{\xi}(\phi(u)) + \phi(\cL_{\xi}u) = -\mathfrak{h}(u) = 0\, ,
\end{split}
\end{equation}
Since by Proposition \ref{prop:propnc} $\mathfrak{h}(u)=0$ if and only if $\mathfrak{h}=0$, we conclude. 
\end{proof}
The Sasakian and K-contact conditions are not equivalent for null contact structures, as the following example clarifies. 
\begin{example}
Consider $M$ to be a connected and simply connected Lie group endowed with a left-invariant global coframe $\{ e^+,e^-,e^2\}$ satisfying:
\begin{equation}
\dd e^+=-a e^+ \wedge e^--e^+\wedge e^2 \, , \quad \dd e^+=e^- \wedge e^2 \, , \quad \dd e^2= e^+ \wedge e^- - a e^- \wedge e^2\,,
\end{equation}
where $a \in \mathbb{R}$. The Lie group structure generated by this coframe is that of the universal cover of the two-dimensional real special linear group $\widetilde{\mathrm{Sl}}(2,\mathbb{R})$. Let us denote by $\{e_+,e_-,e_2\}$ the corresponding dual frame. If we define the following Lorentzian metric on $M$:
\begin{equation}
g=e^+ \odot e^- +e^2\otimes e^2\,,
\end{equation}
then setting $\alpha=e^-$ and $\xi=e_+$ we observe that $(g,\alpha)$ defines a null contact structure. In turn, $\{\xi,e_-,-e_2\}$ defines a light-cone frame. By direct computation:
\begin{equation}
\mathfrak{h}(e_-)=[\xi,\phi(e_-)]-\phi([\xi,e_-])=-\xi+\xi=0\,.
\end{equation}
Hence $(g,\alpha)$ is Sasakian. However, it is not always K-contact, since:
\begin{equation}
(\mathcal{L}_\xi g)(e_-,e_-)=-2g([\xi,e_-],e_-)=-2a g(\xi,e_-)=-2 a\,,
\end{equation}
which is non-zero whenever $a \neq 0$. 
\label{ex:sasnokc}
\end{example}
Although Sasakian null contact structures do not have to be K-contact, the converse turns out to be true.
\begin{proposition}
Let $(g,\alpha)$ be a null K-contact structure $(g,\alpha)$. Then $(g,\alpha)$ is Sasakian.
\label{prop:noknosas}
\end{proposition}
\begin{proof}
Let $\{\xi,u, \phi(u)\}$ be a light-cone frame. By Proposition \ref{prop:propnc}, we just have to check that $g(\mathfrak{h}(u),u)=0$, since this implies that $\mathfrak{h}=0$. Indeed:
\begin{equation}
0=-(\cL_\xi g)(u,\phi(u))=g(\cL_\xi u, \phi(u))+g(u,\mathfrak{h}(u))+g(u,\phi(\cL_\xi u))=g(u,\mathfrak{h}(u))\, ,
\end{equation}
and we conclude. 
\end{proof}
We finish this section by providing the necessary and sufficient condition for a Sasakian null contact structure to be K-contact.
\begin{proposition}
\label{prop:saskc}
Let $(g,\alpha)$ be a Sasakian null contact structure and $\{\xi, u, \phi(u)\}$ be a light-cone frame. Then $(g,\alpha)$ is K-contact if and only if:
\begin{equation}
g(\cL_\xi u, u)=0\,.
\end{equation}
\end{proposition}
\begin{proof}
Since $\cL_\xi \alpha=0$ implies that $\alpha([\xi,u])=0$, then we have that $[\xi,u]=b \xi+c \phi(u)$ for some $b, c \in C^\infty(M)$. Since $(g,\alpha)$ is Sasakian, then $[\xi, \phi(u)]=-c \xi$. Then we observe that $(\cL_\xi g)(\xi,\xi)=(\cL_\xi g)(\xi,u)=(\cL_\xi g)(\xi,\phi(u))=(\cL_\xi g)(u,\phi(u))=(\cL_\xi g)(\phi(u),\phi(u))=0$. Consequently, to guarantee that $\xi$ is Killing we just need to impose:
\begin{equation}
g(\cL_\xi u, u)=0\,,
\end{equation}
and we conclude.
\end{proof}
\begin{remark}
Note that not every Sasakian null contact structure satisfies that $g(\cL_\xi u, u)=0$, as the Example \ref{ex:sasnokc} proves. Hence we explicitly check that the Sasakian condition for null contact structures is weaker than K-contactness. Interestingly enough, this is contrary to the situation for non-null contact structures, for which the Sasaki and K-contact conditions are equivalent in three dimensions whereas in higher dimensions Sasakianity is stronger than the K-contact condition.
\end{remark}


\section{$\varepsilon\eta\,$-Einstein structures and six-dimensional supergravity}

In this last section of the manuscript we introduce the concept of $\varepsilon\eta\,$-Einstein structures, which for $\varepsilon\,$-contact structures with $\varepsilon \neq 0$ are particular cases of the usual notion of $\eta\,$-Einstein structures. Afterwards we will see how that these $\varepsilon\eta\,$-Einstein structures can be used for the construction of solutions of six-dimensional minimal supergravity coupled to tensor multiplet with constant dilaton. 

\begin{definition}
An $\varepsilon\,$-contact structure $(g,\alpha,\varepsilon)$ on a three-dimensional manifold $M$ is said to be $\varepsilon\eta\,$-Einstein if and only if the Ricci curvature tensor $\ric^g$ of $g$ satisfies:
\begin{equation}
\ric^g= \frac{s_g}{2} (\lambda^2+\kappa \varepsilon)g-s_g \kappa \alpha \otimes \alpha\,,
\end{equation}
where $s_g=1$ if $g$ is Riemannian, $s_g=-1$ if $g$ is Lorentzian and $\lambda, \kappa \in \mathbb{R}$ are real constants which satisfy that $\kappa \geq 0$ if $g$ is Lorentzian. We shall refer to $\lambda^2$ and $\kappa$ as the $\varepsilon\eta\,$-Einstein constants.
\end{definition}
\begin{definition}
Let $(M,g)$ be an oriented (pseudo-)Riemannian three-manifold. We denote by $\mathrm{Cont}^{\varepsilon\eta}(M,\varepsilon,\lambda^2,\kappa)$ the space of all $\varepsilon\eta\,$-Einstein structures on $M$ with $\varepsilon\eta\,$-Einstein constants $\lambda^2$ and $\kappa$ and whose Reeb vector field is of norm $\varepsilon$. Similarly, we denote by $\mathrm{Cont}^{\varepsilon\eta}_L(M,\varepsilon,\lambda^2,\kappa)$ $(\mathrm{Cont}^{\varepsilon\eta}_R(M,\lambda^2,\kappa))$ the space of all Lorentzian (Riemannian) $\varepsilon\eta\,$-Einstein structures on $M$ with $\varepsilon\eta\,$-Einstein constants $\lambda^2$ and $\kappa$ and whose Reeb vector field is of norm $\varepsilon$.
\end{definition}
Now we continue by presenting examples of $\varepsilon\eta\,$-Einstein structures, which we will use later to obtain explicit solutions of six-dimensional supergravity. 
\begin{example}
Take the Lie group $\mathrm{SU}(2)$ and consider a left-invariant global frame $\{e_1,e_2,e_3\}$ satisfying the following Lie brackets:
\begin{equation}
[e_2,e_3]=e_1 \, , \quad [e_3,e_1]=\lambda^2 e_2\, , \quad [e_1,e_2]=\lambda^2 e_3\,.
\end{equation}
where $\lambda \neq 0$ is a real constant. Let $\{e^1,e^2,e^3\}$ be the corresponding dual coframe and let us consider the Riemannian metric  $h=e^1\otimes e^1+e^2\otimes e^2+e^3\otimes e^3$ on $\mathrm{SU}(2)$. Let $\alpha_R=e^1$. Then $(h,\alpha_R)$ defines a Riemannian contact metric structure. Also, the Ricci curvature turns out to be:
\begin{equation}
\ric^h=\frac{1}{2}(2\lambda^2-1)h+(1-\lambda^2) \alpha_R \otimes \alpha_R \,.
\end{equation}
Hence we check $(h,\alpha_R) \in \mathrm{Cont}^{\varepsilon\eta}_R(\mathrm{SU}(2),\lambda^2,\lambda^2-1)$ and thus defines an $\varepsilon\eta\,$-Einstein structure on $\mathrm{SU}(2)$.
\label{ex:riem}
\end{example}
\begin{example}
\label{ex:lor}
Let $\widetilde{\mathrm{Sl}}(2,\mathbb{R})$ denote the universal cover of the 2-dimensional special linear group. Consider a left-invariant global frame $\{e_0,e_1,e_2\}$ satisfying the following Lie brackets:
\begin{equation}
[e_0,e_1]=\lambda^2 e_2 \, , \quad [e_0,e_2]=-\lambda^2 e_1 \, , \quad [e_1,e_2]=-e_0\,,
\end{equation}
where $1 \geq \lambda^2 >0$. If $\{e^0,e^1,e^2\}$ is the corresponding dual coframe, let $\chi=-e^0 \otimes e^0 +e^1\otimes e^1+e^2\otimes e^2$ and $\alpha_L=e_0$. Then $(\chi,\alpha_L)$ defines a Lorentzian contact metric structure. Furthermore, since the Ricci curvature of $\chi$ reads:
\begin{equation}
\ric^\chi=-\frac{1}{2}(2\lambda^2-1)\chi+(1-\lambda^2) \alpha_L \otimes \alpha_L\,.
\end{equation}
Therefore $(\chi,\alpha_L)  \in \mathrm{Cont}^{\varepsilon\eta}_L(\widetilde{\mathrm{Sl}}(2,\mathbb{R}),-1,\lambda^2,1-\lambda^2)$, defining thus an $\varepsilon\eta\,$-Einstein structure.
\end{example}
\begin{example}
\label{ex:para}
Take $\widetilde{\mathrm{Sl}}(2,\mathbb{R})$ and let $\{e_0,e_1,e_2\}$ be  a left-invariant global frame satisfying the following Lie brackets:
\begin{equation}
[e_1,e_2]=-\lambda^2 e_0 \, , \quad [e_1,e_0]=-\lambda^2 e_2 \, , \quad [e_2,e_0]=e_1\,,
\end{equation}
where $\lambda^2 \geq 1$.  Let $\{e^0,e^1,e^2\}$ is the corresponding dual coframe and define $\alpha_L=e^1$ and the Lorentzian metric $\chi=-e^0 \otimes e^0 +e^1\otimes e^1+e^2\otimes e^2$. Then $(\chi,\alpha_L)$ defines a para-contact metric structure. On the other hand, the Ricci curvature of $\chi$ turns out to be:
\begin{equation}
\ric^\chi=-\frac{1}{2}(2\lambda^2-1)\chi+ (\lambda^2-1) \alpha_L \otimes \alpha_L\,.
\end{equation}
Consequently $(\chi,\alpha_L)  \in \mathrm{Cont}^{\varepsilon\eta}_L(\widetilde{\mathrm{Sl}}(2,\mathbb{R}),1,\lambda^2,\lambda^2-1)$, so that it is an $\varepsilon\eta\,$-Einstein structure.
\end{example}
\begin{example}
Let us take again $\widetilde{\mathrm{Sl}}(2,\mathbb{R})$ and consider    a left-invariant global frame $\{e_0,e_1,e_2\}$ enjoying the following Lie brackets:
\begin{equation}
[e_1,e_2]=-2 e_0-e_2\, , \quad [e_1,e_0]=e_0\, , \quad [e_2,e_0]=e_1\,.
\end{equation}
If  $\{e^0,e^1,e^2\}$ is the corresponding dual coframe, define $\alpha_L=\alpha_0 (e_0-e_2)$ for $\alpha_0 \neq 0$ and the Lorentzian metric $\chi=-e^0 \otimes e^0 +e^1\otimes e^1+e^2\otimes e^2$. We observe that $(\chi,\alpha_L)$ defines a null contact structure. In particular, the Ricci curvature of $\chi$ takes the form:
\begin{equation}
\ric^\chi=-\frac{1}{2} \chi+\frac{1}{\alpha_0^2}\alpha_L \otimes \alpha_L\,.
\end{equation}
We conclude that  $(\chi,\alpha_L)  \in \mathrm{Cont}^{\varepsilon\eta}_L(\widetilde{\mathrm{Sl}}(2,\mathbb{R}),0,1,\alpha_0^{-2})$ defines an $\varepsilon\eta\,$-Einstein structure.
\label{ex:null}
\end{example}
Our next objective is to describe an explicit procedure to construct solutions of six-dimensional minimal supergravity coupled to a tensor multiplet. We shall begin by defining the configuration space of the theory as well as by specifying which conditions must be satisfied in order to have proper solutions. For that, we will assume in the following that $\tm$ is an oriented six-dimensional manifold.

\begin{definition}
We define the bosonic configuration space of six-dimensional minimal supergravity coupled to a tensor multiplet with constant dilaton on $\tm$ as the set:
\begin{equation}
\mathrm{Conf}(\tm)=\{(\tg,H) \in (\mathrm{Lor}(\tm) \times \Omega^3(\tm)\}\,,
\end{equation}
where $\mathrm{Lor}(\tm)$ denotes the set of Lorentzian metrics on $\tm$.
\end{definition}
\begin{definition}
A pair $(\wh{g},\mathrm{H}) \in \Conf(\tm)$ is a bosonic solution of six-dimensional minimal supergravity coupled to a tensor multiplet with constant dilaton on $\tm$ if:
\begin{equation}
\label{eq:eom}
\ric^{\wh{g}}-\frac{1}{4} \H \circ \H=0\, , \quad \dd \H=\dd \star_{\wh{g}} \H=0\, , \quad \vert \H \vert_{\wh{g}}^2=0\,,
\end{equation}
where $\star_{\wh{g}}$ is the Hodge dual map associated to $\wh{g}$, $\vert \H \vert_{\wh{g}}^2$ denotes the norm of $\H$ with respect to $\wh{g}$ (obtained by contracting indices) and where for any three-forms $\rho, \sigma \in \Omega^3(\tm)$ we have defined the operation:
\begin{equation}
(\rho \circ \sigma) (X,Y)=\wh{g}^{-1} (\iota_X \rho, \iota_Y \sigma)\, , \quad  \forall X,Y \in \mathfrak{X}(\tm)\,.
\end{equation}
We  denote by $\mathrm{Sol}(\tm) \subset \Conf(\tm)$ the set of solutions on $\tm$.
\end{definition}
Now we are in disposition to state the theorem which tells us how to use $\varepsilon\eta\,$-Einstein structures to construct solutions of six-dimensional supergravity.
\begin{theorem}
\label{th:elteorema}
Let $(N,\chi)$ and $(X,h)$ be three-dimensional Lorentzian and Riemannian manifolds, respectively. Let:
\begin{equation}
\begin{split}
(\chi,\alpha_N,\varepsilon_N) \in & \, \, \mathrm{Cont}_L^{\varepsilon\eta}(N,\varepsilon_N,\lambda^2,l^2)\,, \\ (h,\alpha_X) \in & \, \,\mathrm{Cont}_R^{\varepsilon\eta}(X,\lambda^2,\varepsilon_N l^2)\,,
\end{split}
\end{equation}
where $\lambda, l \in \mathbb{R}$. Then, the oriented Cartesian product: manifold
\begin{equation}
\tm=N \times X\,
\end{equation}
carries a family of solutions $(\wh{g}, \H_{\lambda,l}) \in \mathrm{Sol}(\tm)$ of six-dimensional minimal supergravity coupled to a tensor multiplet with constant dilaton and given by:
\begin{equation}
\begin{split}
\wh{g}&=\chi \oplus h\, , \\
 \H_{\lambda,l}&=\lambda \nu_\chi+\frac{l}{3} (\star_\chi \alpha_N) \wedge \alpha_X+\frac{l}{3} \alpha_N \wedge ( \star_h \alpha_X) +\lambda \nu_h\,,
\end{split}
\end{equation}
and parametrized by $(\lambda, l ) \in \mathbb{R}^2$, where $\nu_\chi$ and $\nu_h$ denote the metric volume forms correspondingly. 
\end{theorem}
\begin{proof}
We begin by checking that $\H_{\lambda, l}$ is indeed closed: 
\begin{eqnarray*}
\dd \H_{\lambda, l} = \frac{l}{3}\left ( \dd(\star_\chi \alpha_N)\wedge \alpha_X +  (\star_\chi \alpha_N)\wedge \dd\alpha_X + \dd\alpha_N\wedge (\star_h \alpha_X) -  \alpha_N\wedge \dd (\star_h \alpha_X) \right ) \\ = \frac{l}{3}\left (  (\star_\chi \alpha_N)\wedge \dd\alpha_X +  \dd\alpha_N\wedge (\star_h \alpha_X) \right) = \frac{l}{3} \left ( \star_\chi \alpha_N\wedge \star_h\alpha_X -  \star_{\chi}\alpha_N\wedge \star_h \alpha_X \right) = 0\, ,
\end{eqnarray*}
where we used that $\dd \alpha_N=-\star_\chi \alpha_N$ and $\dd \alpha_X=\star_h \alpha_X$. Now we compute the Hodge dual of $\H_{\lambda,l}$:
\begin{equation}
\star_{\wh{g}} \H_{\lambda,l}=-\lambda \nu_h+\frac{l}{3} \alpha_N \wedge \star_h \alpha_X + \frac{l}{3} \star_\chi \alpha_N \wedge \alpha_X -\lambda \nu_\chi\,,
\end{equation}
where we have used that, for $ \rho \in \Omega^q(N)$ and $\sigma \in \Omega^r(X)$:
\begin{equation}
\star_{\wh{g}}(\rho \wedge \sigma)=(-1)^{r (3-q)} \star_\chi \rho \wedge \star_h \sigma\,.
\end{equation}
Then:
\begin{equation}
\dd \star_{\wh{g}} \H_{\lambda,l}= -\lambda\, \dd\nu_{h}  + \frac{l}{3}\, \dd(\alpha_N\wedge \star_h\alpha_X) + \frac{l}{3} \dd(\star_{\chi} \alpha_N\wedge\alpha_X)  - \lambda\,\dd\nu_{\chi} = 0\, .
\end{equation}
Similarly:
\begin{equation}
\begin{split}
\H_{\lambda,l} \wedge \star_{\wh{g}}  \H_{\lambda,}= \vert \H_{\lambda,l} \vert_{\wh{g}}^2 \nu_\chi \wedge \nu_h&=-\lambda^2 \nu_\chi \wedge \nu_h -\lambda^2 \nu_h \wedge \nu_\chi-\frac{l^2}{9} \vert \alpha_N \vert_\chi^2   \nu_\chi \wedge \nu_h \\&+\frac{l^2}{9} \vert \alpha_N \vert_\chi^2   \nu_\chi \wedge \nu_h=0\,,
\end{split}
\end{equation}
and hence $\vert \H_{\lambda,l} \vert_{\wh{g}}^2 =0$. Finally,  we have to verify that $\ric^{\wh{g}}=\frac{1}{4} \H_{\lambda,l} \circ \H_{\lambda,l}$. For that, we carry out the following computations:
\begin{eqnarray*}
&\H_{\lambda, l}\circ \H_{\lambda, l}\vert_{TN\otimes TN} = \lambda^2 \nu_N\circ \nu_N  + \frac{l^2}{9} (\star_\chi \alpha_N\wedge \alpha_X)\circ (\star_\chi \alpha_N\wedge \alpha_X)\vert_{TN\otimes TN} \\ 
&+ \frac{l^2}{9} (\alpha_N\wedge \star_h \alpha_X)\circ  (\alpha_N\wedge \star_h \alpha_X)\vert_{TN\otimes TN}\, .
\end{eqnarray*}
\noindent
We work out also the following:
\begin{eqnarray*}
& \lambda^2 \nu_N\circ \nu_N = -2\,\lambda^2 \chi\, , \quad (\alpha_N\wedge \star_h \alpha_X)\circ  (\alpha_N\wedge \star_h \alpha_X)\vert_{TN\otimes TN} = 18\,\alpha_N\otimes \alpha_N \,,\\
& (\star_\chi \alpha_N\wedge \alpha_X)\circ (\star_\chi \alpha_N\wedge \alpha_X)\vert_{TN\otimes TN} = 18 (\alpha_N\otimes \alpha_N - \vert \alpha_N\vert^2_{\chi} \,\chi)\, .
\end{eqnarray*}
\noindent
This implies that 
\begin{equation}
\frac{1}{4}\H_{\lambda, l}\circ \H_{\lambda, l}\vert_{TN\otimes TN} = -\frac{\lambda^2}{2}\,\chi -  \frac{l^2}{2}\vert \alpha_N\vert^2_{\chi} \,\chi + l^2 \alpha_N\otimes \alpha_N\, .
\end{equation}
Analogously, we have that:
\begin{equation}
\frac{1}{4}\H_{\lambda, l}\circ \H_{\lambda, l}\vert_{TX\otimes TX} = \frac{\lambda^2}{2}\,h + \frac{l^2}{2}\vert \alpha_N\vert^2_{\chi} \, h - l^2 \vert \alpha_N\vert^2_{\chi} \alpha_X\otimes \alpha_X\, .
\end{equation}
Finally, it can be seen that the mixed components vanish identically:
\begin{equation*}
\H_{\lambda, l}\circ \H_{\lambda, l}\vert_{TN\otimes TX} = \H_{\lambda, l}\circ \H_{\lambda, l}\vert_{TX\otimes TN} = 0\, .
\end{equation*}
Consequently, taking into account that $(\chi,\alpha_N,\varepsilon_N) \in  \, \, \mathrm{Cont}_L^{\varepsilon\eta}(N,\varepsilon_N,\lambda^2, l^2)$ and $ (h,\alpha_X) \in  \, \,\mathrm{Cont}_R^{\varepsilon\eta}(X,\lambda^2,\varepsilon_N l^2)$, we encounter that:
\begin{equation}
\frac{1}{4}\H_{\lambda, l}\circ \H_{\lambda, l}=\frac{1}{4}\H_{\lambda, l}\circ \H_{\lambda, l}\vert_{TN\otimes TN} +\frac{1}{4}\H_{\lambda, l}\circ \H_{\lambda, l}\vert_{TX\otimes TX}=\ric^\chi+ \ric^h\,.
\end{equation} 
Therefore we prove that $(\wh{g}, \H_{\lambda,l}) \in \mathrm{Sol}(\tm)$ and we conclude.
\end{proof}
\begin{remark}
Let $\nabla^\H$ be the unique metric compatible connection on $(\tm, \wh{g})$ with totally skew-symmetric torsion given by $\H \in \Omega^3(\tm)$:
\begin{equation}
\nabla^{\H}=\nabla^{\wh{g}}+\frac{1}{2} \wh{g}^{-1} \H\,,
\end{equation}
being $\nabla^{\wh{g}}$ the Levi-Civita connection associated to $\wh{g}$. Then the Ricci curvature tensor of $\nabla^{\H}$, $\ric(\nabla^{\H})$, is related to that of $\nabla^{\wh{g}}$ by:
\begin{equation}
\ric(\nabla^{\H})=\ric^{\wh{g}}-\frac{1}{4} \H \circ \H=0\,.
\end{equation}
Consequently, under the conditions  and nomenclature of Theorem \ref{th:elteorema}, we conclude via the first equation of (\ref{eq:eom}) that the oriented Cartesian Lorentzian product $(N \times X, \chi \oplus h)$ carries a bi-parametric family of metric-compatible, Ricci-flat connections with totally skew-symmetric, isotropic, closed and co-closed torsion given by $\H_{\lambda,l}$.
\end{remark}
\begin{remark}
Solutions of (\ref{eq:eom}) constructed as the Theorem \ref{th:elteorema} proposes do not need to be supersymmetric (to see the precise definition of supersymmetric solution in this context, we refer the reader to \cite{Akyol:2010iz,Gutowski:2003rg,Angel} and references therein). Therefore, this Theorem gives us a way to obtain generically non-supersymmetric solutions. 
\end{remark}
We finish the document by presenting some explicit examples of solutions of six-dimensional minimal supergravity through the use of Theorem \ref{th:elteorema}.
\begin{example}
Let $(\chi,\alpha_L,-1) \in  \, \, \mathrm{Cont}_L^{\varepsilon\eta}(\widetilde{\mathrm{Sl}}(2,\mathbb{R}),-1,\lambda^2, 1-\lambda^2)$ with $1\geq \lambda^2>0$ be the $\varepsilon\eta\,$-Einstein structure of Example \ref{ex:lor} and $ (h,\alpha_R) \in  \, \,\mathrm{Cont}_R^{\varepsilon\eta}(\mathrm{SU}(2),\lambda^2,\lambda^2-1)$ as in Example \ref{ex:riem}. Then $(\chi \oplus h, \H_{\lambda,l})$, as dictated by Theorem \ref{th:elteorema}, is a solution of six-dimensional minimal supergravity coupled to a tensor multiplet with constant dilaton constructed through the product of a Lorentzian contact metric structure and a Riemannian contact metric structure. 
\end{example}
\begin{example}
Take $(\chi,\alpha_L,1) \in  \, \, \mathrm{Cont}_L^{\varepsilon\eta}(\widetilde{\mathrm{Sl}}(2,\mathbb{R}),1,\lambda^2, \lambda^2-1)$ with $\lambda^2 \geq 1$ as in Example \ref{ex:para} and let $ (h,\alpha_R) \in  \, \,\mathrm{Cont}_R^{\varepsilon\eta}(\mathrm{SU}(2),\lambda^2,\lambda^2-1)$ be the $\varepsilon\eta\,$-Einstein structure of Example \ref{ex:riem}. Then $(\chi \oplus h, \H_{\lambda,l}) \in \mathrm{Sol}(\tm)$, constructed as the product of a para-contact metric structure and a Riemannian contact metric structure. 
\end{example}
\begin{example}
Take $(\chi,\alpha_L,0) \in  \, \, \mathrm{Cont}_L^{\varepsilon\eta}(\widetilde{\mathrm{Sl}}(2,\mathbb{R}),0,1, \alpha_0^{-2})$ with $\lambda^2 \geq 1$ as in Example \ref{ex:null} and $ (h,\alpha_X) \in  \, \,\mathrm{Cont}_R^{\varepsilon\eta}(\mathrm{SU}(2),1,0)$ be the $\varepsilon\eta\,$-Einstein structure of Example \ref{ex:riem} with $\lambda^2 = 1$. Then $(\chi \oplus h, \H_{\lambda,l})$, as prescribed by Theorem \ref{th:elteorema}, is a solution of six-dimensional minimal supergravity coupled to a tensor multiplet with constant dilaton obtained through the product of a null contact metric structure and a Riemannian contact metric structure.
\end{example}

%
%

\end{document}